\documentclass[11pt]{amsart}

\usepackage{amsfonts, amsmath, amscd}
\usepackage[psamsfonts]{amssymb}

\usepackage{amssymb}

\usepackage[usenames]{color}

\newtheorem{theorem}{Theorem}[section]
\newtheorem{lemma}[theorem]{Lemma}
\newtheorem{corollary}[theorem]{Corollary}

\newtheorem{remark}[theorem]{Remark}
\newtheorem{proposition}[theorem]{Proposition}

\numberwithin{equation}{section}

\newcommand{\CC}{C_k}
\newcommand{\NN}{\mathbb{N}}

\newcommand{\w}{\omega}

\newcommand{\Pp}{\mathfrak{P}}

\newcommand{\KK}{\mathcal{K}}

\newcommand{\Nn}{\mathcal{N}}

\newcommand{\AAA}{\mathcal A}
\newcommand{\IR}{\mathbb{R}}
\newcommand{\II}{\mathbb{I}}

\renewcommand{\phi}{\varphi}

\newcommand{\U}{\mathcal U}

\newcommand{\supp}{\mathrm{supp}}

\input xy
\xyoption{all}

\title[Topological properties of function spaces $\CC(X,2)$]{Topological properties of function spaces $\CC(X,2)$ over zero-dimensional metric spaces $X$}

\author{S.~Gabriyelyan}
\address{Department of Mathematics, Ben-Gurion University of the
Negev, Beer-Sheva, P.O. 653, Israel}
\email{saak@math.bgu.ac.il}

\begin{document}

\begin{abstract}
Let $X$ be a zero-dimensional metric space and $X'$ its derived set. We prove the following assertions: (1) the space $\CC(X,2)$ is an Ascoli space iff $\CC(X,2)$ is $k_\IR$-space   iff either $X$ is locally compact or $X$ is not locally compact but $X'$ is compact, (2) $\CC(X,2)$ is a $k$-space iff either $X$ is a topological sum of a Polish locally compact space and a discrete space or $X$ is not locally compact but $X'$ is compact, (3) $\CC(X,2)$ is a sequential space  iff $X$ is a Polish space and either $X$ is locally compact or $X$ is  not locally compact but $X'$ is compact, (4) $\CC(X,2)$ is a Fr\'{e}chet--Urysohn space  iff $\CC(X,2)$ is a Polish space iff $X$ is a Polish  locally compact space, (5) $\CC(X,2)$ is normal iff  $X'$  is separable, (6)  $\CC(X,2)$ has countable tightness iff $X$ is separable. In cases (1)-(3) we obtain also a topological and algebraical structure of $\CC(X,2)$.
\end{abstract}

\maketitle

\section{Introduction}

Topological properties of function spaces  are of great importance  and have been intensively studied from many years (see \cite{Arhangel,kak,mcoy} and references therein). Various topological properties generalizing metrizability are intensively studied by topologists and analysts. Let us mention the Fr\'{e}chet--Urysohn property, sequentiality, $k$-space property,  $k_\IR$-space property and  Ascoli property (all relevant definitions are given in Section \ref{seq:Ascoli-C(X,2)-Pre}). It is known that
\[
\xymatrix{
\mbox{metric} \ar@{=>}[r] & {\mbox{Fr\'{e}chet--}\atop\mbox{Urysohn}} \ar@{=>}[r] & \mbox{sequential} \ar@{=>}[r] &  \mbox{$k$-space} \ar@{=>}[r] &  \mbox{$k_\IR$-space} \ar@{=>}[r] &  {\mbox{Ascoli}\atop\mbox{space}} },
\]
and none of these implications is reversible.

For topological spaces $X$ and $Y$, we denote by $\CC(X,Y)$ the space $C(X,Y)$ of all continuous  functions from $X$ into $Y$ endowed with the compact-open topology. The space $\CC(X,\IR)$ of all real-valued functions on $X$ is denoted by $\CC(X)$. If $G$ is a topological group, then so is $\CC(X,G)$ under the pointwise operation.

For a metrizable space $X$ and $Y=\IR$ or $Y=\II :=[0,1]$, the spaces $\CC(X)$ and $\CC(X,\II)$ are typically not a $k$-space. R.~Pol proved the following remarkable result:
\begin{theorem}[\cite{Pol-1974}] \label{t:Pol-k-space}
Let $X$ be a first countable paracompact space. Then the space $\CC(X,\II)$ is a $k$-space  if and  only if $X=L\cup D$ is the topological sum of a locally compact Lindel\"{o}f space $L$ and a discrete space $D$.
\end{theorem}
It is well known (see \cite{mcoy}) that $\CC(X)$ is metrizable if and only if $X$ is hemicompact, and $\CC(X)$ is completely metrizable if and only if $X$ is a hemicompact $k$-space. Taking into account that $\CC(X,\II)$ is a closed subspace of $\CC(X)$ and the fact that the space $\IR^{\w_1}$ is not a $k$-space by \cite[Problem 7.J(b)]{Kelley}, Theorem \ref{t:Pol-k-space} implies
\begin{corollary} \label{c:Pol-k-space}
For a metric space $X$, the space $\CC(X)$ is a $k$-space  if and  only if $\CC(X)$ is a Polish space if and only if $X$ is a Polish locally compact space.
\end{corollary}

Metric spaces $X$ for which $\CC(X)$ and $\CC(X,\II)$  are  $k_\IR$-spaces or  Ascoli spaces were completely characterized in \cite{GKP}: $\CC(X)$ is an Ascoli space  iff $\CC(X,\II)$ is an Ascoli space  iff $\CC(X)$ is a $k_\IR$-space iff $\CC(X,\II)$ is a $k_\IR$-space iff $X$ is locally compact.

If $Y=2=\{ 0,1\}$ is the doubleton, the situation changes. Recall that the derived set $X'$ of a topological space $X$ is the set of all non-isolated points of $X$. G.~Gruenhage et al. proved the following result.
\begin{theorem}[\cite{GTZ}] \label{t:GTZ-sequen}
Let $X$ be a zero-dimensional Polish space. Then $\CC(X,2)$ is sequential if and only if $X$ is either locally compact or the derived set $X'$ is compact.
\end{theorem}

T.~Banakh and S.~Gabriyelyan \cite{BG} posed the following problem: Characterize separable metrizable spaces $X$ for which the function space $\CC(X,2)$ is an Ascoli space. We prove the following result which also shows that the Fr\'{e}chet--Urysohn property, sequentiality and $k$-space property differ on spaces of the form $\CC(X,2)$. While these properties coincide for $\CC(X)$ by Pytkeev's theorem \cite{Pyt2}: for a Tychonoff space $X$, the space $\CC(X)$ is a $k$-space if and only if it is Fr\'{e}chet--Urysohn.
\begin{theorem} \label{t:Ascoli-C(X,2)-A}
Let $X$ be  a zero-dimensional metric space $X$. Then:
\begin{enumerate}
\item[{\rm (i)}]  $\CC(X,2)$ is  Ascoli if and only if $\CC(X,2)$ is a $k_\IR$-space if and only if either  $X$ is locally compact or  $X$ is not locally compact but the derived set $X'$ is compact;
\item[{\rm (ii)}] $\CC(X,2)$ is a $k$-space if and only if either $X=L\cup D$ is a topological sum of a separable metrizable locally compact space $L$ and a discrete space $D$ or $X$ is not locally compact but the derived set $X'$ is compact;
\item[{\rm (iii)}] $\CC(X,2)$ is a sequential space  if and only if $X$ is a Polish space and either $X$ is locally compact or $X$ is  not locally compact but the derived set $X'$ is compact;
\item[{\rm (iv)}] $\CC(X,2)$ is a Fr\'{e}chet--Urysohn space  if and only if $\CC(X,2)$ is a Polish space if and only if $X$ is a Polish  locally compact space.
\end{enumerate}
\end{theorem}
In Theorems  \ref{t:Ascoli-C(X,2)-Ascoli}--\ref{t:Ascoli-C(X,2)-seq} below we obtain also a  topological and algebraic structure of function spaces $\CC(X,2)$ for cases (i)-(iii).
Note that (iii) of Theorem \ref{t:Ascoli-C(X,2)-A} generalizes Theorem \ref{t:GTZ-sequen} by showing that the assumption on the space $X$ to be Polish can be omitted: the sequentiality of  $\CC(X,2)$  implies that $X$ is a Polish space. Let us remark also that our prove of (iii) essentially differs from the proof of Theorem \ref{t:GTZ-sequen}.

For topological spaces $X$ and $Y$, we denote by $C_p(X,Y)$ the space $C(X,Y)$ endowed with the topology of pointwise convergence.  If $Y=\II$, R.~Pol proved the following theorem.
\begin{theorem}[\cite{Pol-1974}] \label{t:Pol-Ck-normal}
For a metrizable space $X$, the following assertions are equivalent: (i) $\CC(X,\II)$ is normal, (ii) $C_p(X,\II)$ is normal, (iii) $\CC(X,\II)$ is Lidel\"{o}f, (iv) $X'$ is separable.
\end{theorem}
It turns out that the same holds also for $Y=2$ with an identical proof.
\begin{theorem} \label{t:Ascoli-C(X,2)-N}
For a zero-dimensional metric space $X$, the following assertions are equivalent: (i) $\CC(X,2)$ is normal, (ii) $C_p(X,2)$ is normal, (iii) $\CC(X,2)$ is Lidel\"{o}f, (iv) $X'$ is separable.
\end{theorem}

In the next theorem we characterize metric spaces $X$ for which  $\CC(X,2)$ has countable tightness.
\begin{theorem} \label{t:Ascoli-C(X,2)-Tight}
For a zero-dimensional metric space $X$ the following assertions are equivalent: (i) $\CC(X,2)$ has countable tightness, (ii) $\CC(X,2)$ is a $\Pp_0$-space, (iii) $C_p(X,2)$ has countable tightness, (iv) $X$ is separable.
\end{theorem}

\section{Auxiliary results} \label{seq:Ascoli-C(X,2)-Pre}

We start from the definitions of the following well-known notions. A topological space $X$ is called
\begin{itemize}
\item[$\bullet$] {\em Fr\'{e}chet-Urysohn} if for any cluster point $a\in X$ of a subset $A\subset X$ there is a sequence $\{ a_n\}_{n\in\NN}\subset A$ which converges to $a$;
\item[$\bullet$] {\em sequential} if for each non-closed subset $A\subset X$ there is a sequence $\{a_n\}_{n\in\NN}\subset A$ converging to some point $a\in \bar A\setminus A$;
\item[$\bullet$] a {\em $k$-space} if for each non-closed subset $A\subset X$ there is a compact subset $K\subset X$ such that $A\cap K$ is not closed in $K$;
\item[$\bullet$] a {\em $k_\IR$-space} if  a real-valued function $f$ on $X$ is continuous if and only if its restriction $f|_K$ to any compact subset $K$ of $X$ is continuous.
\end{itemize}

For topological spaces $X$ and $Y$, denote by $\psi: X\times\CC(X,Y)\rightarrow Y$, $\psi(x,f):=f(x)$, the evaluation map. Recall that a subset $\KK$ of $\CC(X,Y)$ is {\em evenly continuous} if the restriction of $\psi$ onto $X\times \KK$ is jointly continuous, i.e. for any $x\in X$, each $f\in\KK$  and every neighborhood $O_{f(x)}\subset Y$ of $f(x)$ there exist neighborhoods $U_f\subset \KK$ of $f$ and $O_x\subset X$ of $x$ such that $U_f(O_x):=\{g(y):g\in U_f,\;y\in O_x\}\subset O_{f(x)}$.
Following \cite{BG}, a regular (Hausdorff) space $X$ is called an {\em Ascoli space} if each compact subset $\KK$ of $\CC(X)$ is evenly continuous.
It is easy to see that a space $X$ is Ascoli if and only if the canonical valuation map $X\hookrightarrow \CC(\CC(X))$ is an embedding, see \cite{BG}.  By Ascoli's theorem \cite[3.4.20]{Eng}, each $k$-space is Ascoli. N.~Noble \cite{Noble} proved that any $k_\IR$-space is Ascoli.

For a sequence $\{ G_n \}_{n\in \NN}$  of  groups, the \emph{direct sum} of $G_n$ is denoted by
\[
\bigoplus_{n\in \NN} G_n :=\left\{ (g_n)_{n\in \NN} \in \prod_{n\in \NN} G_n : \; g_n = e_n \mbox{ for almost all } n \right\}.
\]

Let $(G,\tau)$ be a topological group. The filter of all open neighborhoods of the identity $e$ is denoted by $\mathcal{N}(G)$.  Let $\{ (G_n, \tau_n) \}_{n\in \NN}$ be a sequence of (Hausdorff)  topological groups. For every $n\in \NN$ fix $U_n \in \mathcal{N}(G_n)$ and put
\[
 \prod_{n\in \NN} U_n :=\left\{ (g_n)_{n\in \NN} \in \prod_{n\in \NN} G_n : \; g_n \in U_n \mbox{ for  all } n\in \NN \right\}.
\]
Then the sets of the form $ \prod_{n\in \NN} U_n $, where $U_n \in \mathcal{N}(G_n)$ for every $n\in \NN$, form a neighborhood basis at the unit of a (Hausdorff) group topology $\mathcal{T}_b$ on $\prod_{n\in \NN} G_n$ that is called  {\em the box topology}. The items (i) and (ii) of the next proposition are well-known.
\begin{proposition} \label{p:Ascoli-C(X,2)-Box}
Let $\{ G_n\}_{n\in \NN}$ be a sequence of compact topological groups and let $G:= \big( \bigoplus_{n\in \NN} G_n, \mathcal{T}_b \big)$. Then:
\begin{enumerate}
\item[{\rm (i)}] the group $G$ is a $k$-space (even a $k_\w$-space);
\item[{\rm (ii)}] if $G_n$ is metrizable for every $n\in \NN$, then $G$ is a sequential space;
\item[{\rm (iii)}] if $G_n$ is metrizable for every $n\in \NN$, then $G$ is a Fr\'{e}chet--Urysohn space if and only if $G$ is a separable locally compact metrizable space if and only if $G_n$ is finite for all but finitely many $n$.
\end{enumerate}
\end{proposition}
\begin{proof}
(i) It is easy to see that the topology $\mathcal{T}_b$ of $G$ is defined by the sequence $\{ \prod_{i\leq n} G_i\}_{n\in\NN}$ of compact groups. So $G$ is  a $k_\w$-space (in particular, $G$ is hemicompact).
(ii) follows from (i) and \cite[Lemma 1.5]{ChMPT} (note that every compact subset of $G$ is contained in $\prod_{i\leq n} G_i$ for some $n\in\NN$). (iii) follows from Proposition 3.18 and Corollary 2.17 of \cite{GKL}.
\end{proof}

If $G_n =G$ for every $n\in\NN$, we set $G^\infty := \big( \bigoplus_{n\in \NN} G_n, \mathcal{T}_b \big)$.
Proposition \ref{p:Ascoli-C(X,2)-Box} immediately implies the next result.
\begin{corollary}\label{c:C(X,2)}
For every infinite metrizable compact group $G$ the group $G^\infty$ (in particular, $(2^\w)^\infty$) is a sequential non-Fr\'{e}chet--Urysohn space.
\end{corollary}

Recall that the family of subsets
\[
[C;\epsilon]:= \{ f\in \CC(X): |f(x)|<\epsilon \; \forall x\in C\},
\]
where $C$ is a compact subset of $X$ and $\epsilon >0$, forms a basis of open neighborhoods at the zero function $\mathbf{0} \in\CC(X)$.

We need the following generalization of the countably infinite metric fan $M$. Fix a sequence $\pmb{\kappa} =(\kappa_n)_{n\in\NN}$ of (non-zero) cardinal numbers. Denote by $M_{\pmb{\kappa}}$ the space
\[
M_{\pmb{\kappa}}:= \left( \bigcup_{n\in\NN} \kappa_n \times\{ n\}\right) \cup \{\infty\},
\]
where the points of $\bigcup_{n\in\NN} \kappa_n \times\{ n\}$ are isolated, and the basic neighborhoods of $\infty$  are
\[
U(n)=\left( \bigcup_{i\geq n} \kappa_i \times\{ i\}\right) \cup \{\infty\}, \quad n\in\NN.
\]
If $\alpha\times \{n\}\not= \beta\times\{ m\}$ and $n\leq m$, set
\[
\rho\big( \alpha\times \{n\}, \beta\times\{ m\}\big) = \rho\big( \alpha\times \{n\},\infty \big) =\frac{1}{n}.
\]
Then $\rho$ is a complete metric on $M_{\pmb{\kappa}}$. So the space $M_{\pmb{\kappa}}$ is a complete metrizable space. It is easy to see that $M_{\pmb{\kappa}}$ is locally compact at $\infty$ if and only if all but finitely many of cardinals $\kappa_n$ are finite. If $\kappa_n =\w$ for every $n\in\NN$, the space $M_{\pmb{\kappa}} =(\w\times\NN)\cup\{ \infty\}$ is called the {\em countably infinite metric fan} and is denoted by $M$; so $M$ is a Polish space which is not  locally compact at $\infty$. Set
\[
\CC^0\big( M_{\pmb{\kappa}}, 2\big):= \left\{ f\in \CC\big( M_{\pmb{\kappa}}, 2\big): \; f(\infty)=0 \right\}.
\]
So $\CC^0\big( M_{\pmb{\kappa}}, 2\big)$ is a clopen subgroup of $\CC\big( M_{\pmb{\kappa}}, 2\big)$ and $\CC\big( M_{\pmb{\kappa}}, 2\big)=\mathbb{Z}(2)\times \CC^0\big( M_{\pmb{\kappa}}, 2\big)$.

\begin{proposition} \label{p:Ascoli-M-k}
Let $\pmb{\kappa} =(\kappa_n)_{n\in\NN}$ be a sequence of cardinal numbers. Then:
\begin{enumerate}
\item[{\rm (i)}] The group $\CC\big( M_{\pmb{\kappa}}, 2\big)$ is topologically isomorphic to the direct sum $\mathbb{Z}(2) \oplus \bigoplus_{n\in\NN} 2^{\kappa_n}$ endowed with the box topology.
\item[{\rm (ii)}] $\CC\big( M_{\pmb{\kappa}}, 2\big)$ is a $k_\w$-space.
\item[{\rm (iii)}] $\CC\big( M_{\pmb{\kappa}}, 2\big)$ is a sequential space if and only if $\kappa_n \leq\w$ for every $n\in\NN$.
\item[{\rm (iv)}] The following assertions are equivalent:
\begin{enumerate}
\item[{\rm (iv${}_1$)}]
$\CC\big( M_{\pmb{\kappa}}, 2\big)$ is a Fr\'{e}chet--Urysohn space;
\item[{\rm (iv${}_2$)}] $\CC\big( M_{\pmb{\kappa}}, 2\big)$ is a locally compact Polish  abelian  group;
\item[{\rm (iv${}_3$)}] $\kappa_n \leq\w$ for every $n\in\NN$ and $\kappa_n $ is finite for almost all indices $n$.
\end{enumerate}
\end{enumerate}
\end{proposition}

\begin{proof}
(i): It is enough to show that $\CC^0\big( M_{\pmb{\kappa}},2\big)$ is topologically isomorphic to the direct sum $G:= \bigoplus_{n\in\NN} 2^{\kappa_n}$ endowed with the box topology.
Define the map $F: \CC^0\big( M_{\pmb{\kappa}},2\big) \to  G$ by the rule
\[
F(f)=(f_n), \mbox{ where } f_n :=f|_{\kappa_n \times\{ n\}}, \; n\in\NN.
\]
We claim that $F$ is a desired topological isomorphism.

First we note that $F$ is well-defined. Indeed, since $f$ is continuous, there is $m\in\NN$ such that $f(U(m))=\{ 0\}$. So $f_n =0$ for every $n\geq m$, and $F(f)$ belongs to $\bigoplus_{n\in\NN} 2^{\kappa_n}$. Clearly, $F$ is an algebraic isomorphism.

Let $K$ be a compact subset of $M_{\pmb{\kappa}}$. Then, for every $n\in\NN$, the intersection $K_n :=K\cap \big(\kappa_n \times\{ n\}\big)$ is also compact in the discrete space $\kappa_n \times\{ n\}$, so $K_n$ is finite. Hence every compact subset  of $M_{\pmb{\kappa}}$ is contained in a compact subset of $M_{\pmb{\kappa}}$ of the form
\[
K=\left( \bigcup_{n\in\NN} K_n \times \{ n\}\right) \cup \{\infty\},
\]
where $K_n $ is a finite subset of $\kappa_n $ for every  $n\in\NN.$ Since for $0<\epsilon <1$
\[
[K;\epsilon] =\left\{ f\in \CC^0\big( M_{\pmb{\kappa}},2\big): \; f(K)=\{ 0\} \right\},
\]
we see that $F([K;\epsilon]) = \bigoplus_{n\in\NN} 2^{\kappa_n \setminus K_n}$. Taking into account that the sets $\bigoplus_{n\in\NN} 2^{\kappa_n \setminus K_n}$ form an open basis at $0$ in the group $G$, this equality means that $F$ is a homeomorphism. Thus $F$ is a topological isomorphism.

(ii)  follows from (i) and Proposition \ref{p:Ascoli-C(X,2)-Box}(i).

(iii): If $G$ is sequential, then all compact groups $2^{\kappa_n}$ are sequential. So $\kappa_n \leq\w$ for every $n\in\NN$, as a compact abelian group is sequential if and only if it is metrizable.  Conversely, if  $\kappa_n \leq\w$ for every $n\in\NN$, then  $G$ is sequential by Proposition \ref{p:Ascoli-C(X,2)-Box}(ii).

(iv) follows from (i), (iii) and Proposition \ref{p:Ascoli-C(X,2)-Box}(iii).
\end{proof}

We shall use the following sufficient condition on a space $X$ to be a non-Ascoli space.
\begin{proposition}[\cite{GKP}] \label{p:Ascoli-sufficient}
Assume  a Tychonoff space $X$ admits a  family $\U =\{ U_i : i\in I\}$ of open subsets of $X$, a subset $A=\{ a_i : i\in I\} \subset X$ and a point $z\in X$ such that
\begin{enumerate}
\item[{\rm (i)}] $a_i\in U_i$ for every $i\in I$;
\item[{\rm (ii)}] $\big|\{ i\in I: C\cap U_i\not=\emptyset \}\big| <\infty$  for each compact subset $C$ of $X$;
\item[{\rm (iii)}] $z$ is a cluster point of $A$.
\end{enumerate}
Then $X$ is not an Ascoli space.
\end{proposition}

We need also the following construction from \cite{PolSmen}.
Denote by $\mathfrak{s}=\{ 0,1/2,1/3,\dots\}\subset\IR$ the convergent sequence and set
\[
L:=M\cup (\NN\times \mathfrak{s}).
\]
Then $L$ is a zero-dimensional Polish space which is not locally compact at the point $\infty\in M$. For every $p,q\in\NN$, set
\[
a_{p,q} :=(p,p+q)\in M, \quad b_{p,q}:=\big(p,1/(p+q)\big)\in\NN\times \mathfrak{s},\quad c_p:= (p,0)\in\NN\times \mathfrak{s}, 
\]
and define the function
\[
f_{p,q}(a_{p,q})=f_{p,q}(b_{p,q}):=1, \quad \mbox{ and } f_{p,q}(x):=0 \mbox{ if } x\not\in\{ a_{p,q},b_{p,q}\}.
\]
To use Proposition \ref{p:Ascoli-sufficient} we define also the open neighborhood $U_{p,q}$ of $f_{p,q}$ in $\CC(L,2)$ by
\[
U_{p,q} :=\{ h\in\CC(L,2): h(a_{p,q})=h(b_{p,q})=1, \; h(\infty)=h(c_p)=0 \},
\]
and set $A:=\big\{ f_{p,q}: p,q\in\NN\big\}$ and $\U :=\big\{ U_{p,q}: p,q\in\NN\big\}$.
\begin{lemma}\label{l:Ascoli-L}
The set $A$ and the family $\U$  satisfy the following conditions:
\begin{enumerate}
\item[{\rm (i)}] the zero function $\mathbf{0}$ is a unique cluster point of $A$;
\item[{\rm (ii)}] if $K\subset \CC(L,2)$ is compact, the set $\{ (p,q)\in\NN\times\NN : U_{p,q} \cap K \not=\emptyset\}$ is finite.
\end{enumerate}
\end{lemma}
\begin{proof}
(i) It is easy to see that any compact subset $C$ of $L$ is contained in a compact subset of $L$ of the form
\[
C_{t,n}=\big(\left\{ (p,s)\in M: s\leq t(p), p\in\NN\right\} \cup\{ \infty\}\big) \cup \bigcup_{i\leq n} \big(\{ i\}\times \mathfrak{s}\big),
\]
where $n\in\NN$ and $t:\NN \to \NN$ is a function. If we take $q>t(n+1)$, then $f_{n+1,q}\in [C;\epsilon]\cap A$ for every $\epsilon>0$. Thus $\mathbf{0}\in \overline{A}$.

Denote by $\supp(g)=\{ x\in L: g(x)\not= 0\}$ the support of a function $g\in\CC(L,2)$.
Note that the supports of the functions $f_{p,q}\in A$ are pairwise disjoint. So $f_{p,q}$ is an isolated point of $A$. Now if $g$ is a cluster point of $A$ and $g\not\in A$, then $g$ is in the closure of $A$ in the topology $\tau_p$ of pointwise convergence (note that $\tau_p$ is metrizable since $L$ is countable). Thus $g$ must be the zero function $\mathbf{0}$.

(ii) Suppose for a contradiction that the set $J:=\{ (p,q)\in\NN\times\NN : U_{p,q} \cap K \not=\emptyset\}$ is infinite for some compact subset $K\subset \CC(L,2)$. Note that the space $\CC(L,2)$ is a $\sigma$-space and hence $K$ is metrizable (see \cite{gruenhage}). So there is a sequence $F=\big\{ h_{p_i,q_i}\big\}_{i\in\NN} \subset K$ converging to a function $h\in K$, where $h_{p_i,q_i}\in U_{p_i,q_i}$ and all pairs $(p_i,q_i)$ are distinct. In particular, $h(\infty)=0$.

We claim that the sequence $(p_i)$ is bounded. Indeed, assuming the converse we would find $p_{i_1}<p_{i_2}<\dots$ such that the points $a_{p_{i_s},q_{i_s}}$ converge to $\infty$. But since $h_{p_{i_s},q_{i_s}}\big( a_{p_{i_s},q_{i_s}}\big) =1$ and $h_{p_{i_s},q_{i_s}}(\infty)=0$ for every $s\in\NN$, this contradicts the equicontinuity of the compact set $F\cup\{ h\}$ on the compact set $\big\{ a_{p_{i_s},q_{i_s}}: s\in\NN\big\} \cup\{\infty\}$.

So without loss of generality we can assume that $p_i = n$ for every $i\in\NN$. Set $C:= \{ b_{n,q}, c_n: q\in\NN\}$ and $\epsilon=1/2$. Since $h_{n,q_i}(b_{n,q_i})=1$ and $h_{n,q_i}(c_n)=0$ for every $i\in\NN$, we obtain that $h(c_n)=0$ and the compact set $F\cup\{ h\}$ is not equicontinuous on the compact set $C$. This contradiction shows that the set $J$ is finite.
\end{proof}
Lemma \ref{l:Ascoli-L}   and  Proposition \ref{p:Ascoli-sufficient} immediately imply
\begin{proposition}[\cite{Pol-2015}] \label{p:Ascoli-L-Pol}
The space $\CC(L,2)$ is not Ascoli.
\end{proposition}


\section{Proofs of Theorems \ref{t:Ascoli-C(X,2)-A}--\ref{t:Ascoli-C(X,2)-Tight}}


We start from the following proposition proved by R.~Pol.
\begin{proposition}[\cite{Pol-2015}] \label{p:Ascoli-C(X,2)-Pol}
Let $X$ be  a zero-dimensional metric space. If the space $\CC(X,2)$ is  Ascoli, then either $X$ is locally compact or $X$ is not locally compact but the derived set $X'$ is compact.
\end{proposition}
\begin{proof}
Suppose for a contradiction that $X$ is not locally compact and $X'$ is not compact.  By Lemma 8.3 of \cite{Douwen}, $X$ contains as a closed subspace the countable metric fan $M$. Since $X'$ is not compact, $X$ contains also an isomorphic copy of $\NN \times \mathfrak{s}$. Thus $X$ contains a closed subspace $Y$ which is homeomorphic to the space $L$. As $Y$ is a retract of $X$ (see \cite{Engel-1969}  for a more general assertion) and $X$ has the Dugundji extension property by \cite{borges},  $\CC(Y,2)$ is a retract of $\CC(X,2)$. So $\CC(X,2)$ is not Ascoli by Proposition \ref{p:Ascoli-L-Pol} and \cite[Proposition 5.2]{BG}.
\end{proof}
Below and in the proof of Theorem \ref{t:Ascoli-C(X,2)-seq} we shall use the following fact:  a separable metrizable space $X$ such that $X'$ is compact is Polish (this easily follows from Cantor's theorem \cite[4.3.8]{Eng}).
The next result gives a partial answer to Problem 6.8 in \cite{BG}.
\begin{corollary}[\cite{Pol-2015}] \label{c:Ascoli-C(X,2)-Pol}
Let $X$ be  a zero-dimensional separable metric space. Then the space $\CC(X,2)$ is  Ascoli if and only if either $X$ is locally compact  or $X$ is not locally compact but the derived set $X'$ is compact.
\end{corollary}
\begin{proof}
The necessity follows from Proposition \ref{p:Ascoli-C(X,2)-Pol}, and the sufficiency follows from Corollary \ref{c:Pol-k-space} and Theorem \ref{t:GTZ-sequen}.
\end{proof}

In the next theorem we strengthen Theorem \ref{t:Ascoli-C(X,2)-A}(i) by showing also the algebraic structure of the group  $\CC(X,2)$ which is essentially used in the proof of Theorem \ref{t:Ascoli-C(X,2)-k-space} below.
\begin{theorem} \label{t:Ascoli-C(X,2)-Ascoli}
For   a zero-dimensional metric space $X$, the following assertions are equivalent:
\begin{enumerate}
\item[{\rm (1)}] the space $\CC(X,2)$ is  Ascoli;
\item[{\rm (2)}] the space  $\CC(X,2)$ is a $k_\IR$-space;
\item[{\rm (3)}] one of the following conditions holds
\begin{enumerate}
\item[{\rm (3a)}] $X$ is locally compact; in this case $\CC(X,2)$ is the product of a family of Polish abelian groups;
\item[{\rm (3b)}] $X$ is not locally compact but the derived set $X'$ is compact; in this case the open subgroup
    \[
    H:= \{ f\in \CC(X,2): \; f|_{X'} \equiv 0\}
    \]
    of $\CC(X,2)$ is a $k$-space and is topologically isomorphic to the direct sum $\bigoplus_{n\in\NN} 2^{\kappa_n}$ endowed with the box topology for some sequence $\pmb{\kappa} =(\kappa_n)_{n\in\NN}$ of cardinal numbers.
\end{enumerate}
\end{enumerate}
\end{theorem}

\begin{proof}
(1)$\Rightarrow$(3): Proposition \ref{p:Ascoli-C(X,2)-Pol} implies that  either $X$ is locally compact or $X$ is not locally compact but the derived set $X'$ is compact.

 Assume that $X$ is locally compact. Then $X$ is a disjoint union of a family $\{ X_i\}_{i\in I}$ of separable metrizable locally compact spaces by \cite[5.1.27]{Eng}. So
\[
\CC(X,2) = \prod_{i\in I} \CC(X_i,2),
\]
where each space $\CC(X_i,2)$ is Polish by Corollary \ref{c:Pol-k-space}. 

Assume that $X$ is not locally compact but the derived set $X'$ is compact. As in the first paragraph of the proof of Theorem 3.1 in \cite{GTZ}, there is a clopen outer base $\{ U_n\}_{n\in\NN}$ of $X'$ such that $U_1 =X$, $U_{n+1} \subset U_n$, and $U_n\setminus U_{n+1}$ is infinite for all $n\in\NN$. Noting that $U_n\setminus U_{n+1}$  is a clopen discrete subspace of $X$, we set $\pmb{\kappa}=(\kappa_n)$, where $\kappa_n =|U_n\setminus U_{n+1}|$ for every $n\in\NN$. Let $T:X\to M_{\pmb{\kappa}}$ be a map such that $T(X')=\{\infty\}$ and $T|_{U_n\setminus U_{n+1}}$ is an injective map onto $\kappa_n\times \{n\}$.  Then $T$ is a continuous map such that $T^{-1}(K)$ is compact for every compact subset $K\subset M_{\pmb{\kappa}}$ (see the structure of compact subsets of $M_{\pmb{\kappa}}$ given in the proof of Proposition \ref{p:Ascoli-M-k}). The map
\[
T^\ast: \CC^0(M_{\pmb{\kappa}},2) \to H, \quad T^\ast(f):=f\circ T,
\]
is easily seen to be a continuous isomorphism of abelian groups. As $T^\ast\big( [K;\varepsilon]\big) =[T^{-1}(K); \varepsilon]$ for every compact subset $K\subset M_{\pmb{\kappa}}$ and $0<\varepsilon<1$, $T^\ast$ is open and hence it is a topological isomorphism. Thus, by Proposition \ref{p:Ascoli-M-k}, $H$ is a $k$-space which is topologically isomorphic to the direct sum $\bigoplus_{n\in\NN} 2^{\kappa_n}$ endowed with the box topology. Therefore also $ \CC(X,2)$ is a $k$-space.

(3)$\Rightarrow$(2): If $X$ is locally compact, then $\CC(X,2) = \prod_{i\in I} \CC(X_i,2)$ is a $k_\IR$-space by  \cite[Theorem 5.6]{Nob}. The second case is immediate.

(2)$\Rightarrow$(1) follows from \cite{Noble}.
\end{proof}


\begin{theorem} \label{t:Ascoli-C(X,2)-k-space}
Let $X$ be a zero-dimensional metric space. Then the space $\CC(X,2)$ is a $k$-space if and only if one of the following conditions holds
\begin{enumerate}
\item[{\rm (a)}] $X$ is a topological sum of a separable metrizable locally compact space $L$ and a discrete space $D$; in this case $\CC(X,2)$ is the product of a  Polish abelian group and a compact abelian group;
\item[{\rm (b)}] $X$ is not locally compact but the derived set $X'$ is compact, in this case $\CC(X,2)$ contains an open subgroup which is topologically isomorphic to the direct sum $\bigoplus_{n\in\NN} 2^{\kappa_n}$ endowed with the box topology for some sequence $\pmb{\kappa} =(\kappa_n)_{n\in\NN}$ of cardinal numbers.
\end{enumerate}
\end{theorem}

\begin{proof}
Assume that  $\CC(X,2)$ is a $k$-space. Then,  by Theorem \ref{t:Ascoli-C(X,2)-Ascoli},  either $X$ is locally compact or $X$ is not locally compact but $X'$ is compact. In the second case $\CC(X,2)$ is a $k$-space and (b) holds by (3b) of  Theorem \ref{t:Ascoli-C(X,2)-Ascoli}.

If $X$ is locally compact, then  $X$ is a disjoint union of a family $\{ X_i\}_{i\in I}$ of separable metrizable locally compact spaces by \cite[5.1.27]{Eng}. So
\[
\CC(X,2) = \prod_{i\in I} \CC(X_i,2),
\]
where each space $\CC(X_i,2)$ is Polish  by Corollary \ref{c:Pol-k-space}.

We claim that $X_i$ is discrete for all but countably many $i\in I$. Indeed, suppose for a contradiction that $X_i$ is not discrete for an uncountable subset $J$ of $I$. 
So $X_i$ contains a clopen infinite compact subset $K_i$, $i\in J$. As $K_i$ is compact and metric, $\CC(K_i,2)$ is discrete and countable, so it is homeomorphic to $\NN$. Since $\CC(K_i,2)$ is homeomorphic to a closed subset of $\CC(X_i,2)$, we obtain that $\CC(X,2)$ contains a closed subspace $Z$ which is homeomorphic to $\NN^{|J|}$. As $Z$ is not a $k$-space by \cite[2.7.16]{Eng}, we get a contradiction.

Denote by $L$ the direct topological sum of all non-discrete spaces $X_i$ and by $D$ the direct topological sum of all discrete spaces $X_i$ (if they exist). So $L$ is a Polish locally compact space, and $X= L\cup D$ is a topological union of $L$ and $D$. Then
\[
\CC(X,2) = \CC(L,2) \times 2^D,
\]
where $\CC(L,2) $ is a Polish abelian group group by Corollary \ref{c:Pol-k-space}. 
 and $2^D$ is a compact abelian group.

Conversely, assume that (a) holds and $X=L\cup D$, where $L$ is a separable metrizable locally compact space and $D$ is a discrete space. Then  $\CC(X,2) =\CC(L,2) \times 2^D$ is a $k$-space by \cite[3.3.27]{Eng}. If (b) holds and $X$ is not locally compact but the derived set $X'$ is compact, then $\CC(X,2)$ is a $k$-space by Theorem \ref{t:Ascoli-C(X,2)-Ascoli}(3b).
\end{proof}

\begin{theorem}  \label{t:Ascoli-C(X,2)-seq}
Let $X$ be a zero-dimensional metric space. Then $\CC(X,2)$ is  sequential if and only if $X$ is a Polish space and one of the following conditions holds
\begin{enumerate}
\item[{\rm (a)}] $X$ is not locally compact but the derived set $X'$ is compact, in this case $\CC(X,2)$ has an open subgroup $H$ which is topologically isomorphic to $(2^\w)^\infty$ and $\CC(X,2)$ is homeomorphic to $(2^\w)^\infty$;
\item[{\rm (b)}] $X$ is  locally compact, in this case $\CC(X,2)$ is a  Polish space.
\end{enumerate}
\end{theorem}

\begin{proof}
Assume that $\CC(X,2)$ is a sequential space.  By (the proof of) Theorem \ref{t:Ascoli-C(X,2)-k-space}, we have to cases.

{\em Case (a)}: $X$ is not locally compact but the set $X'$ is compact. So $\CC(X,2)$ contains an open subgroup $H$ which is topologically isomorphic to the direct sum $\bigoplus_{n\in\NN} 2^{\kappa_n}$ endowed with the box topology for some sequence $\pmb{\kappa} =(\kappa_n)_{n\in\NN}$ of cardinal numbers, see (3b) of Theorem \ref{t:Ascoli-C(X,2)-Ascoli}. Since every $2^{\kappa_n}$ is also sequential, we obtain that $\kappa_n \leq\w$ for every $n\in\NN$. By the proof of the implication (3)$\Rightarrow$(2) of Theorem \ref{t:Ascoli-C(X,2)-Ascoli}, this means that $X\setminus X'$ is countable. Thus $X$ is Polish. Moreover, for infinitely many indices $n$, the cardinal $\kappa_n := |U_n\setminus U_{n+1}|$ is equal to $\w$ since, otherwise, the space $X$ would be locally compact.  Since the groups $2^\w$ and $\mathbb{Z}(2)^r \times 2^\w$ are topologically isomorphic for every natural number $r$, we obtain that the group $H$ is topologically isomorphic to $(2^\w)^\infty$.  Further, since $X$ is Polish, $\CC(X,2)$ is separable. Hence the discrete group $S:= \CC(X,2)/H$ is a countable abelian group of order $2$. Thus $S=\oplus_{i\in\NN} S_i$, where $S_i =\mathbb{Z}(2)$ for every $i\in\NN$ (see \cite[11.2]{Fuchs}). Therefore $\CC(X,2)$ is homeomorphic to $S\times (2^\w)^\infty$. Since the groups $S\times (2^\w)^\infty$, $(\mathbb{Z}(2)\times 2^\w)^\infty$ and $(2^\w)^\infty$ are topologically isomorphic, the item (a) is proven.

{\em Case (b)}: $X$ is locally compact and $\CC(X,2) = \CC(L,2) \times 2^D$, where $L$ is a Polish locally compact space and $D$ is a discrete space. Since $\CC(X,2)$ is sequential, $D$ is countable. 
So $X$ is a Polish locally compact space and $\CC(X,2)$ is a Polish group.
\end{proof}

Now we prove Theorem \ref{t:Ascoli-C(X,2)-A}.

{\em Proof of Theorem \ref{t:Ascoli-C(X,2)-A}}.
Items (i)-(iii) follow from Theorems \ref{t:Ascoli-C(X,2)-Ascoli}--\ref{t:Ascoli-C(X,2)-seq}, respectively. Let us prove (iv).

If $\CC(X,2)$  is Fr\'{e}chet--Urysohn, then $X$ is a Polish locally compact space and $\CC(X,2)$  is Polish by Theorem \ref{t:Ascoli-C(X,2)-seq} and Corollary \ref{c:C(X,2)}. If $X$ is a Polish locally compact space, then $\CC(X,2)$ is Fr\'{e}chet--Urysohn by Corollary \ref{c:Pol-k-space}. 
$\Box$

For  (metrizable) Tychonoff spaces $X$ and $Y$, it would be interesting to obtain an analogue of Theorem \ref{t:Ascoli-C(X,2)-A} for the spaces $C_p(X,2)$ and $C_p(X,Y)$.

To prove Theorem \ref{t:Ascoli-C(X,2)-N} we need the following lemma whose proof is identical with the proof of Lemma 1 in \cite{Pol-1974}. For the sake of completeness we prove this lemma.
\begin{lemma} \label{l:Ascoli-C(X,2)-Pol}
Let $X$ be a zero-dimensional metric space such that $X'$ is not Lindel\"{o}f. Then $\CC(X,2)$ and $C_p(X,2)$ contain $\NN^{\w_1}$ as a closed set.
\end{lemma}
\begin{proof}
As the pointwise topology is weaker than the compact-open one, it is enough to prove that $C_p(X,2)$ contains $\NN^{\w_1}$ as a closed set. Since $X$ is paracompact and zero-dimensional and $X'$ is not Lindel\"{o}f, there exists a family, discrete in the space $X$,  of clopen sets $\{ F_i : i<\w_1\}$ and a family $\{ z_i : i<\w_1\}$ such that $z_i \in F_i\cap X'$ for every $i<\w_1$.
Let
\[
A_i :=\big\{ f\in \CC(X,2): f(X\setminus F_i)=\{ 0\} \big\}
\]
and
\[
 A:=\left\{ f\in \CC(X,2): f(X\setminus \bigcup_{i\in I} F_i)=\{ 0\} \right\}.
\]
Then $A=\prod_{i\in I} A_i$, and since $A$ is closed in $C_p(X,2)$, it is enough to prove that every $A_i$ contains a discrete closed countable subset.
Fix arbitrarily $i<\w_1$. Take a one-to-one sequence $\{ x_{n,i}: n\in\NN\} \subset F_i$ such that $x_{n,i}\to z_i$ and $x_{n,i}\not= z_i$ for every $n\in\NN$, and chose a function $f_{n,i}\in A_i$ such that
\[
f_{n,i}(x_{m,i}) =\left\{
\begin{split}
0, & \; m>n,\\
1, & \; m\leq n,
\end{split}
\right. \quad f_{n,i} (z_i)=0. 
\]
Clearly, the open neighborhood $\big\{ f\in A_i: f(x_{n,i})=1, f(x_{n+1,i})=0\big\}$ of $f_{n,i}$ does not contain $f_{k,i}$ for $k\not= n$. So the set $B:=\{ f_{n,i}: n\in\NN\}$ is discrete in $A_i$. Let us show that $B$ is also closed in $A_i$.

Assuming the converse we choose $h\in \overline{B}\setminus B$. Then $h(z_i)=0$ and the open neighborhood $h+\big[ \{ x_{1,i},\dots,x_{n,i}\}; 1/2\big]$ of $h$ contains infinitely many $f_{m,i}\in B$. By the construction of $f_{n,i}$ we obtain that $h(x_{n,i})=1$ for every $n\in\NN$, and hence $h(z_i)=1$. This contradiction shows that $B$ is closed in $A_i$.
\end{proof}

We are ready to prove  Theorem \ref{t:Ascoli-C(X,2)-N}.

{\em Proof of Theorem \ref{t:Ascoli-C(X,2)-N}.}
The implications (i)$\Rightarrow$(iv) and (ii)$\Rightarrow$(iv) follow from Lemma \ref{l:Ascoli-C(X,2)-Pol} and from the fact that the space $\NN^{\w_1}$ is not normal by \cite[2.7.16]{Eng}. The implication (iv)$\Rightarrow$(iii) follows from Proposition 1 of \cite{Pol-1974}. The implications (iii)$\Rightarrow$(i) and (iii)$\Rightarrow$(ii) are clear.
$\Box$

Recall that a topological space $X$ has {\em countable tightness at a point $x\in X$} if whenever $x\in \overline{A}$ and $A\subseteq X$, then $x\in \overline{B}$ for some countable $B\subseteq A$; $X$ has {\em countable tightness} if it has countable tightness at each point $x\in X$. 

Recall also (see \cite{Mich})  that a topological space $X$ is called {\em cosmic}, if $X$ is a regular space with a countable network (a family $\mathcal{N}$ of subsets of $X$ is called a \emph{network} in $X$ if,
whenever $x\in U$ with $U$ open in $X$, then $x\in N\subseteq U$ for some $N\in\mathcal{N}$). It is trivial that any cosmic space has countable tightness.

Following \cite{Banakh}, a family $\mathcal{N}$ of subsets of a topological space $X$ is called a  {\em Pytkeev network at a point $x\in X$} if $\Nn$ is a network at $x$ and for every open set $U\subseteq X$ and a set $A$ accumulating at $x$ there is a set $N\in\Nn$ such that $N\subseteq U$ and $N\cap A$ is infinite. The space $X$ is called a a {\em $\Pp_0$-space} if $X$ has a countable Pytkeev network. Any $\Pp_0$-space is cosmic.

Now we are ready to prove Theorem \ref{t:Ascoli-C(X,2)-Tight}

{\it Proof of Theorem \ref{t:Ascoli-C(X,2)-Tight}.}
(i)$\Rightarrow$(iv) and (iii)$\Rightarrow$(iv): Suppose for a contradiction that $X$ is not separable. Then $X$ has a discrete family $\AAA :=\{ A_i\}_{i\in\w_1}$ of clopen subsets by \cite[4.1.15 and 5.1.12]{Eng} (recall that $X$ is zero-dimensional).  Define the monomorphism $T:2^{\w_1} \to C(X,2)$ by
\[
T\big( (z_i)_{i< \w_1}\big) := \sum_{i<\w_1} z_i \chi_{A_i},
\]
where $\chi_A$ is the characteristic function of a subset $A$ of $X$. Clearly, $T$ is continuous as a map from the compact group $2^{\w_1}$ into $C_p(X,2)$.  Also $T$ is an embedding from  $2^{\w_1}$ into $\CC(X,2)$. Indeed, since $\AAA$ is discrete, any compact subset $K$ of $X$ intersects only with a finite subfamily $\{ A_{i_1}, \dots, A_{i_m}\}$ of $\AAA$. Then, for every $\epsilon >0$, we obtain
\[
T\left( \prod_{k=1}^m \{ 0_{i_k} \} \times 2^{\w_1 \setminus \{ i_1,\dots,i_m\}} \right) \subseteq [K;\epsilon] \subset \CC(X,2).
\]
Thus $T:2^{\w_1} \to \CC(X,2)$ is an embedding. 

 So $\CC(X,2)$ and $C_p(X,2)$ contain an isomorphic copy of the compact abelian group $2^{\w_1}$. As $2^{\w_1}$ is not metrizable, $2^{\w_1}$ and hence also $\CC(X,2)$ and $C_p(X,2)$ have uncountable tightness by Corollary 4.2.2 in \cite{ArT}. This contradiction shows that $X$ is separable.

(iv)$\Rightarrow$(ii): If $X$ is separable, then $\CC(X,2)$ is a $\Pp_0$-space by \cite{Banakh} (see also \cite[Corollary 6.4]{GK-GMS1}). (ii)$\Rightarrow$(i) is clear, see \cite{GK-GMS1}.  (iv)$\Rightarrow$(iii):  By Proposition 10.4 of \cite{Mich}, the space  $C_p(X,2)$ is cosmic, so it has countable tightness.
$\Box$

\begin{remark}{\em
R.~Pol and F.~Smentek \cite{PolSmen} proved the following interesting result: If $X$ is a zero-dimensional realcompact $k$-space, then the group $\CC(X,2)$ is reflexive. Note also that, for the group of rational numbers $\mathbb{Q}\subset\IR$ and the convergent sequence $\mathfrak{s}$,  the metric spaces $\mathbb{Q}\times \w_1$ and $\mathfrak{s}\times \w_1$ are realcompact by \cite[3.11.5 and 5.5.10(b)]{Eng}. These results and Theorems \ref{t:Ascoli-C(X,2)-A} and \ref{t:Ascoli-C(X,2)-N} show the following: (i) the group $\CC(\mathbb{Q}\times \w_1,2)$ is  reflexive but it is neither normal nor Ascoli, (ii) the group $\CC(\mathbb{Q},2)$ is a reflexive $\Pp_0$-group but is not Ascoli, (iii) the group $\CC(\mathfrak{s}\times \w_1,2)$ is a reflexive $k_\IR$-space but is not normal.}
\end{remark}

\vspace{3mm}
{\bf Acknowledgments}.
The author is deeply indebted to Professor R.~Pol who sent  to T.~Banakh and me a solution of Problem 6.8 in \cite{BG} (see Proposition \ref{p:Ascoli-C(X,2)-Pol} and  Corollary \ref{c:Ascoli-C(X,2)-Pol}). In \cite{Pol-2015},   R.~Pol  noticed that it can be shown that the space $\CC(L,2)$ contains a closed countable non-Ascoli subspace using some ideas from \cite{PolSmen}. Using this fact  R.~Pol  proved that the space $\CC(L,2)$ is not Ascoli (see Proposition \ref{p:Ascoli-L-Pol}). We provide another proof of the last result but also essentially using constructions from \cite{PolSmen}.

\bibliographystyle{amsplain}

\end{document}